\documentclass[12pt]{article}

\usepackage{amsfonts}
\usepackage{amsmath}

\newcommand{\beq}{\begin{equation}}
\newcommand{\eeq}{\end{equation}}
\newcommand{\bea}{\begin{eqnarray}}
\newcommand{\eea}{\end{eqnarray}}
\newcommand{\beas}{\begin{eqnarray*}}
\newcommand{\eeas}{\end{eqnarray*}}

\newtheorem{theorem}{Theorem}[section]
\newtheorem{definition}[theorem]{Definition}
\newtheorem{proposition}[theorem]{Proposition}

\newtheorem{lemma}[theorem]{Lemma}
\newtheorem{remark}[theorem]{Remark}
\newtheorem{example}[theorem]{Example}
\newtheorem{examples}[theorem]{Examples}
\newtheorem{foo}[theorem]{Remarks}
\newtheorem{assumption}[theorem]{Assumption}

\newenvironment{proof}{\addvspace{\medskipamount}\par\noindent{\it
Proof}.}
{\unskip\nobreak\hfill$\Box$\par\addvspace{\medskipamount}}

\title{Convergence of the financial value of weak information for a sequence of discrete-time markets}
\author{Geoff Lindsell}
\date{\today}

\begin{document}
\maketitle
\begin{abstract}
We examine weak anticipations in discrete-time and continuous-time financial markets consisting of one risk-free asset and multiple risky assets, defining a minimal probability measure associated with the anticipation that does not depend on the choice of a utility function.  We then define the financial value of weak information in the discrete-time economies and show that these values converge to the financial value of weak information in the continuous-time economy in the case of a complete market.
\end{abstract}

\begin{section}{Introduction}
A typical assumption in financial mathematics is the existence of portfolio consisting of one risk-free asset, and multiple risky assets, measurable to a common filtration satisfying the usual conditions.  That is, along a finite time horizon $[0,T]$, we assume there is a $d$-dimensional vector $S_{t} \in \mathbb{R}^{d} \text{, } S_{t} \in \mathcal{F}_{t}$, where $(S_{t})_{0 \leq t \leq T}$ denotes the price process of the various assets, and $\mathcal{F} = (\mathcal{F}_{t})_{0 \leq t \leq T}$ is the "public" information flow available to the "ordinary investor." 
\\
We extend this assumption to the existence of an "insider" or "informed investor" that possesses, from the beginning, additional information about the outcome of a some random variable Y.  One approach to model this information is using an enlarged filtration $\mathcal{G} = (\mathcal{G}_{t})_{0 \leq t \leq T}$ where $\mathcal{G}_{t} = \underset{\epsilon > 0}{\bigcap}\left(\mathcal{F}_{t + \epsilon} \vee \sigma(Y)\right)$.  See, for example, the works of Amendinger \cite{MR1775229, MR1954386, MR1632213}.
\\
In this paper, we use an alternative approach to model this additional information.  Rather than considering the enlarged filtration, we assume that the insider has knowledge of the law of a functional Y of the price process, rather than the exact value of Y as the insider would in the enlargement approach.
\\
In doing so, we follow a multiple step program. Firstly, we define the minimal probability measure associated with weak information in discrete-time economies as well as the continuous economy. Secondly, we define the financial value of weak information in discrete-time economies as well as the continuous time economies. Lastly, we show that, under certain conditions, the financial value of weak information in discrete-time economies converges to the financial value of weak information in the continuous time economy.
\\
The financial value of weak information is expressed in the language of utility functions. As in the works of Kramkov and Schachermayer, \cite{MR1722287}, we have the following definition for utility functions.
\begin{definition}
A \textbf{utility function} is a strictly increasing, strictly concave, and twice continuously differentiable function
\begin{align*}
U: (0, + \infty) \rightarrow \mathbb{R}
\end{align*}
which satisfies the \textbf{Inada conditions}
\begin{align*}
\lim_{x \rightarrow + \infty} U'(x) = 0, \lim_{x \rightarrow 0^{+}} U'(x) = + \infty
\end{align*}
where we use the convention that $U(x) = - \infty$ for $x \leq 0$.  We shall denote by $I$ the inverse of $U'$, and by $\tilde{U}$ the convex conjugate of $U$:
\begin{align*}
\tilde{U}(y) = \max_{x > 0}(U(x) - xy)
\end{align*}
That is, the Fenchel-Legendre transform of $-U(-x)$.
\end{definition}
\end{section}

\begin{section}{Weak Anticipation and the Financial Value of Weak Information in the Discrete-Time Economy}
\begin{subsection}{Mathematical Framework}
In this section we consider discrete-time economies of the form
\begin{align*}
\left(\Omega_{n}, \mathcal{H}^{n} = (\mathcal{H}_{m})_{0 \leq m \leq n}, P_{n}, S^{n} = (S_{m})_{0 \leq m \leq n} \right)
\end{align*}
where $\Omega_{n}$ represents a discrete sample space, $\mathcal{H}^{n}$ represents the discrete filtration, $P_{n}$ represents a probability measure on $\Omega_{n}$ and $S^{n}$ is the price process in this discrete economy. We will assume that $S_{m} \in \mathbb{R}^{d}$ and that one of the assets is a risk-free asset with r denoting the risk free rate of return of the firsk-free asset. Further, we assume $P_{n}(\omega) > 0$ for each $\omega \in \Omega_{n}$ \\
Let $V_{n}$ represent the value of the discrete-time portfolio at time $n$, where the initial wealth of the investor, $V_{0} = x$, for fixed $x \in \mathbb{R}$. \\ 
Throughout this work, $\mathcal{F}^{n} = (\mathcal{F}_{m})_{0 \leq m \leq n}$ will represent the filtration of the ordinary investor.  That is, $\mathcal{F}^{n}$ will represent the public information flow where $S^{n} \in \mathcal{F}^{n}$. \\
To enlarge this filtration for the informed insider, let $\mathcal{P}_{n}$ be a discrete set containing the range of $Y_{n}$, endowed with the discrete $\sigma$-algebra and let $Y_{n}: \Omega_{n} \rightarrow \mathcal{P}_{n}$ be an $\mathcal{F}_{n}$-measurable random variable.  we shall denote by $\mathcal{G}^{n}$ the filtration $\mathcal{F}^{n}$ initially enlarged with $Y_{n}$, i.e. $\mathcal{G}_{m} = \mathcal{F}_{m} \vee \sigma (Y_{n})$.  Denote by $\nu_{n}$ the law of $Y_{n}$ and assume that $Y_{n}$ admits a regular disintegration with respect to the filtration $\mathcal{F}^{n}$. \\
Let us now introduce some common definitions, seen, for example, in Baudoin, Nguyen-Ngoc \cite{MR2213259}
\begin{definition}
The space $\mathcal{M}_{\mathcal{H}^{n}}(S^{n})$ of martingale measures in the discrete-time economy is the set of probability measures $\tilde{P}_{n} \sim P_{n}$ such that $(S_{m})_{0 \leq m \leq n}$ is an $\mathcal{H}^{n}$-adapted martingale under $\tilde{P}_{n}$.
\end{definition} 
\begin{definition}
We say that there is no arbitrage on the financial market 
\begin{align*}
(\Omega, (\mathcal{H}_{m})_{0 \leq m \leq n}, (S_{m})_{0 \leq m \leq n}, P_{n})
\end{align*}
if there exists a probability measure $\tilde{P}_{n}$ equivalent to $P_{n}$ such that $(S_{m})_{0 \leq m \leq n}$ is an $\mathcal{H}^{n}$-adapted martingale under $\tilde{P}_{n}$.
\end{definition}
If we assume our market is free from arbitrage, then we may assume $\mathcal{M}_{\mathcal{H}^{n}}(S^{n})$ is nonempty.  For a complete market, $\mathcal{M}_{\mathcal{H}^{n}}(S^{n})$ is a singleton, say $\mathcal{M}_{\mathcal{H}^{n}}(S^{n}) = \{\tilde{P}_{n}\}$, where $\tilde{P}_{n}$ is the unique probability measure under which discounted stock prices are martingales (reference \cite{MR2049045}, for example).  We will assume complete markets, i.e $\mathcal{M}_{\mathcal{F}^{n}}(S^{n})$ is a singleton, in what follows.
\begin{definition}
A \textbf{trading strategy} is defined as a stochastic process
\begin{align*}
\Theta_{m} = ((\Theta_{m}^{0},\Theta_{m}^{1},...,\Theta_{m}^{d}))_{0 \leq m \leq n}
\end{align*}
in $\mathbb{R}^{d+1}$, where $\Theta_{m}^{i}$ denotes the number of shares of an asset i held in the portfolio at time m.  We assume the sequence $\Theta$ is predictable in the sense that $\Theta_{0}^{i} \in \mathcal{H}_{0}$, and $\Theta_{m}^{i} \in \mathcal{H}_{m-1}$ for $m \geq 1$.
\end{definition}
\begin{definition}
The \textbf{value of the portfolio} at time m is the scalar product
\begin{align*}
V_{m}(\Theta_{m}) = \Theta_{m} \cdot S_{m} = \sum_{i = 0}^{d}\Theta_{m}^{i}S_{m}^{i}
\end{align*}
Its \textbf{discounted value} is $\tilde{V}_{m}(\Theta_{m}) = \frac{1}{S_{m}^{0}}(\Theta_{m} \cdot S_{m})$
\end{definition}
\begin{definition}
A strategy is called \textbf{self-financing} if the following equation is satisfied for all $m \in \{0,1,...,n-1\}:$
\begin{align*}
\Theta_{m} \cdot S_{m} = \Theta_{m+1} \cdot S_{m}
\end{align*}
The interpretation is that an investor readjusts the position from $\Theta_{m}$ to $\Theta_{m+1}$ without bringing or consuming any wealth.
\end{definition}
\begin{definition}
A strategy $\Theta$ is \textbf{admissible} if it is self-financing and if $V_{m}(\Theta_{m}) \geq 0$ for any $m \in \{0,1,...,n\}.$
\end{definition}
\begin{proposition}
The value process $(V_{m}(\Theta_{m}))_{0 \leq m \leq n}$ is a martingale for any admissible strategy $\Theta_{m}$ and any equivalent martingale measure $\tilde{P}_{n} \in \mathcal{M}_{\mathcal{H}^{n}}(S^{n})$.
\end{proposition}
\begin{proof}
Let $E_{\tilde{P}_{n}}$ denote the conditional expecation with respect to $\tilde{P}_{n}$.  Then using the predictability of $\Theta_{m}$, the martingale property of $S_{m}$, and the self-financing property, respectively we see
\begin{align*}
E_{\tilde{P}_{n}}[V_{m+1}(\Theta_{m+1}) | \mathcal{F}_{m}] &= E_{\tilde{P}_{n}}[\Theta_{m+1} \cdot S_{m+1} | \mathcal{F}_{m}] \\
&= \Theta_{m+1} E_{\tilde{P}_{n}}[S_{m+1} | \mathcal{F}_{m}] \\
&= \Theta_{m+1} \cdot S_{m} \\
&= \Theta_{m} \cdot S_{m} \\
&= V_{m}(\Theta_{m})
\end{align*}
\end{proof}
It is common in the literature, e.g. in the works of Amendinger \cite{MR1775229, MR1954386, MR1632213} or Baudoin \cite{MR1919610, MR2021790, MR2213259} to make the following set of assumptions when we consider the strong information case.
\begin{assumption}
There exists a jointly measurable and $\mathcal{F}^{n}$-adapted process 
\begin{align*}
\eta_{m}^{y,n}, \text{ } 0 \leq m < n, \text{ } y \in \mathcal{P}_{n}
\end{align*}
satisfying for all $0 \leq m < n$ and $y \in \mathcal{P}_{n}$,
\begin{align*}
P[Y_{n} = y | \mathcal{F}_{m}] = \eta_{m}^{y,n}P[Y_{n} = y]
\end{align*}
\end{assumption}
\begin{remark}
We can note here that, if we denote by $P_{n}^{y}$ the disintegrated probability measure defined by $P_{n}^{y} = P_{n}[\cdot | Y_{n} = y]$, then the above assumption implies that for $m <  n$,
\begin{align*}
P_{/\mathcal{F}_{m}}^{y} = \eta_{m}^{y,n}P_{/\mathcal{F}_{m}}
\end{align*}
In particular, for each $y \in \mathcal{P}_{n}$, the process $(\eta_{m}^{y})_{0 \leq m < n}$ is a martingale in the filtration $\mathcal{F}^{n}$.
\end{remark}
However, in this work we will not need the previous assumption of the existence of a distintegrated probability measure which tells us the anticipation path-by-path.  Instead, we study the distribution of the anticipation, which characterizes the weak information case.  More precisely, we have the following definition.
\begin{definition}
The probability measure 
\begin{align*}
P_{n}^{\nu_{n}}(\omega) = \sum_{y \in \mathcal{P}_{n}} \tilde{P}_{n}[\omega | Y_{n} = y]\nu_{n}[Y_{n} = y]
\end{align*}
is called the \textbf{minimal probability measure associated with weak information} $(Y_{n}, \nu_{n})$, where $\tilde{P}_{n} \in \mathcal{M}_{\mathcal{F}^{n}}(S^{n})$ is an equivalent martingale measure.
\end{definition}
By Proposition 3.1 of \cite{MR3954302} (see below), $P_{n}^{\nu_{n}}$ is minimal in the set of probability measures $Q_{n}$ to $P_{n}$ such that $Q_{n}[Y_{n} = y] = \nu_{n}[Y_{n} = y]$ for all $y \in \mathcal{P}_{n}$.  We denote this set by $\mathcal{E}^{\nu_{n}}$.  In particular we have,
\begin{proposition}
If $\phi$ is a convex function, then
\begin{align*}
\underset{Q_{n} \in \mathcal{E}^{\nu_{n}}}{\min} E_{\tilde{P}_{n}}\left[ \phi \left( \frac{dQ_{n}}{d\tilde{P}_{n}} \right) \right] =  E_{\tilde{P}_{n}}\left[ \phi \left( \frac{dP_{n}^{\nu_{n}}}{d\tilde{P}_{n}} \right) \right]
\end{align*}
where $\frac{dQ_{n}}{d\tilde{P}_{n}}$ denotes the Radon-Nikodym derivative of $Q_{n}$ with respect to $\tilde{P}_{n}$.
\end{proposition}
In this work we want to distinguish between the insider's anticipation and the uninformed investor at the terminal time.  Since we focus on maximizing utility of wealth rather than the value of wealth, we have the following definition.
\begin{definition}
The \textbf{financial value of weak information} is the lowest expected utility that can be gained from anticipation.  That is,
\begin{align*}
u(x, \nu_{n}) = \underset{Q_{n} \in \mathcal{E}^{\nu_{n}}}{\min}\underset{\Theta \in \mathcal{A}_{\mathcal{G}^{n}}(S^{n})}{\max}E_{Q_{n}}[U(V_{n})]
\end{align*}
\end{definition}
\cite{MR3954302} gives us a closed form expression for the financial value of weak information. Namely, we have the following result.
\begin{theorem}
The financial value of weak information in a complete market is
\begin{align*}
u(x, \nu_{n}) = \underset{\Theta \in \mathcal{A}_{\mathcal{G}^{n}}(S^{n})}{\max}E_{P_{n}^{\nu_{n}}}[U(V_{n})] = E_{P_{n}^{\nu_{n}}}\left[U \left( I\left( \frac{\Lambda_{n}(x)}{(1 + r)^{n}} \right) \right) \frac{d\tilde{P}_{n}}{dP_{n}^{\nu_{n}}}\right]
\end{align*}
where $\Lambda_{n}(x)$ is determined by
\begin{align*}
E_{\tilde{P}_{n}} \left[ \frac{1}{(1+r)^{n}} I \left( \frac{\Lambda_{n}(x)}{(1+r)^{n}} \frac{d\tilde{P}_{n}}{dP_{n}^{\nu_{n}}} \right) \right] = x
\end{align*}
where $\tilde{P}_{n} \in \mathcal{M}_{\mathcal{F}^{n}}(S^{n})$ is the unique probability measure under which the price processes are martingales. 
\end{theorem}
\end{subsection}
\begin{subsection}{Trinomial and Multinomial Transition Probabilities under Weak Anticipation}
In the next few sections we aim to prove that the random walk under weak anticipation is Markov in the trinomial and multinomial case.  In the trinomial case, we consider a sample sample space $\Omega = \{u, m, d\}$; the idea is that at each iteration of the $n$-th discrete time economy, the stock can either go "up," "down," or take a "middle" path, and we denote these three possibilities by u, d, m, respectively. \\
At each transition, the value of the risk-free asset rises by a factor of $(1+r)$, and the risky asset in the $k$-th period has payoffs $aS_{k-1}$ if the stock goes up in the case u, $cS_{k-1}$ if the stock goes down in the case d, and $bS_{k-1}$ moves by a factor in between a and c in the case m.  Using this notation, the transition probabilities satisfy the following.
\begin{proposition}
Let $l\in\{1,...,n-1\}$, $j\in\{0,...,n-l\}$, and $i\in\{0,...,n-l\}$. Then

\begin{align*}
\ P_{n}^{\nu_{n}}[S_{n-l+1}=aS_{n-l}\mid  &S_{n-l}=a^{i}b^{j}c^{n-i-j}s]\\
&= \frac{\sum\limits_{h=i}^{l+i+j-1} \sum\limits_{k=j}^{l+i+j-1-h}\frac{\binom{l - 1} {h - i - 1, k - j, l + i + j - h - k}}{\binom{n}{h, k, n-h-k}}}{\sum\limits_{h=i}^{l+i+j} \sum\limits_{k=j}^{l+i+j-h}\frac{\binom{l } {h - i, k - j, l + i + j - h - k}}{\binom{n}{h, k, n-h-k}}}\nu_{h,k}\\
\end{align*}
\end{proposition}

\begin{proof} Note that
$P_{n}^{\nu_{n}}[S_{n-l+1} = aS_{n-l}\mid S_{n-l} = a^{i}b^{j}c^{n-l-i-j}s] = \frac{P_{n}^{\nu_{n}}[S_{n-l+1} = aS_{n-l}, S_{n-l} = a^{i}b^{j}c^{n-l-i-j}s]}{P_{n}^{\nu_{n}}[S_{n-l} = a^{i}b^{j}c^{n-l-i-j}s]}$
We compute
\begin{align*}
\ P_{n}^{\nu_{n}}[S_{n-l} = a^{i}b^{j}c^{n-l-i-j}s] &= \sum\limits_{h = 0}^{n} \sum\limits_{k=0}^{n}P_{n}[S_{N-l} = a^{i}b^{j}c^{n-l-i-j}s \mid S_{n} = a^{h}b^{k}c^{n-h-k}s]\nu_{h,k}\\
&= \sum\limits_{h = 0}^{n} \sum\limits_{k=0}^{n} \frac{P_{n}[S_{n-l} = a^{i}b^{j}c^{n-l-i-j}s, S_{n} = a^{h}b^{k}c^{n-h-k}s]}{P_{n}[S_{n} = a^{h}b^{k}c^{n-h-k}s]} \nu_{h,k}\\
&=\sum\limits_{h=i}^{l+i+j} \sum\limits_{k = j}^{l + i + j - h} \frac{\binom{h}{h-i,k-j,l+i+j-h-k}\tilde{p}_{1}^{h}\tilde{p}_{2}^{k}\tilde{p}_{3}^{n-h-k}}{\binom{n}{h, k, N-h-k}\tilde{p}_{1}^{h}\tilde{p}_{2}^{k}\tilde{p}_{3}^{n-h-k}} \nu_{h,k}\\
&= \sum\limits_{h=i}^{l+i+j} \sum\limits_{k = j}^{l + i + j - h} \frac{\binom{h}{h-i,k-j,l+i+j-h-k}}{\binom{n}{h, k, n-h-k}} \nu_{h,k}
\end{align*}

where $\tilde{p}_{1}, \tilde{p}_{2}, \tilde{p}_{3}$ represent the risk neutral probabilities and we have used the fact that we need $i \leq h, j \leq k, n-l-i-j \leq n-h-k$ for the numerators to be nonzero.  Next, 

$P_{n}^{\nu_{n}}[S_{n-l+1} = aS_{n-l}, S_{n-l} = a^{i}b^{j}c^{n-l-i-j}s]  = $
\begin{align*}
 ... \\ &= \sum\limits_{h=0}^{n} \sum\limits_{k=0}^{n} \frac{P_{n}[S_{n-l+1} = aS_{n-l}, S_{n-l} = a^{i}b^{j}c^{n-l-i-j}s, S_{n} = a^{h}b^{k}c^{n-h-k}s]}{P_{n}[S_{n} = a^{h}b^{k}c^{n-h-k}s]} \nu_{h,k} \\ &= \sum\limits_{h = i}^{l + i + j - 1} \sum\limits_{k = j}^{l + i + j - 1 - h} \frac{\binom{l-1}{h-i-1, k-j, l+i+j-h-k}}{\binom{n}{h,k,N-h-k}} \nu_{h,k}\\
\end{align*}
\\
Thus, \\
$P_{n}^{\nu_{n}}[S_{n-l+1} = aS_{n-l} \mid S_{n-l} = a^{i}b^{j}c^{n-l-i-j}s] = \frac{\sum\limits_{h = i}^{l + i + j - 1} \sum\limits_{k = j}^{l + i + j - 1 - h} \frac{\binom{l-1}{h-i-1, k-j, l+i+j-h-k}}{\binom{n}{h, k, n-h-k}}}{\sum\limits_{h=i}^{l+i+j} \sum\limits_{k = j}^{l + i + j - h} \frac{\binom{h}{h-i,k-j,l+i+j-h-k}}{\binom{n}{h,k,n-h-k}}} \nu_{h,k}$
\\
\\
\end{proof}
We can use similar techniques to prove the following.
\begin{proposition}
Let $l\in\{1,...,n-1\}$, and $i_{j}\in\{0,...,n-l\}$ where $\sum_{1}^{k} i_{j} = n-l$. Then

\begin{align*}
\ P_{n}^{\nu_{n}}[S_{n-l+1}=a_{1}S_{n-l}\mid  &S_{n-l}=a_{1}^{i_{1}} \cdots a_{k}^{i_{k}}s]\\
&= \frac{\sum\limits_{\forall x \in [1,k], i_{x} \leq j_{x} \leq l - 1 - \sum\limits_{y = 1}^{k} i_{y} - \sum\limits_{y=1}^{x-1} j_{y}} \frac{\binom{l - 1} {j_{1} - i_{1} - 1, j_{2} - i_{2}, ... , j_{k} - i_{k}}}{\binom{n}{j_{1}, ..., j_{k}}}}{\sum\limits_{\forall x \in [1,k], i_{x} \leq j_{x} \leq l - 1 - \sum\limits_{y = 1}^{k} i_{y} - \sum\limits_{y=1}^{x-1} j_{y}} \frac{\binom{l } {j_{1} - i_{1}, j_{2} - i_{2}, ... , j_{k} - i_{k}}}{\binom{n}{j_{1}, ..., j_{k}}}}\nu_{j_{1},...,j_{k}}\\
\end{align*}
\end{proposition}

\begin{proof} Note that
$P_{n}^{\nu_{n}}[S_{n-l+1} = a_{1}S_{n-l}\mid S_{n-l} = a_{1}^{i_{1}} \cdots a_{k}^{i_{k}}s] = \frac{P_{n}^{\nu_{n}}[S_{n-l+1} = a_{1}S_{n-l}, S_{n-l} = a_{1}^{i_{1}} \cdots a_{k}^{i_{k}}s]}{P_{n}^{\nu_{n}}[S_{n-l} = a_{1}^{i_{1}} \cdots a_{k}^{i_{k}}s]}$
We compute
\begin{align*}
\ P_{n}^{\nu_{n}}[S_{n-l} = a_{1}^{i_{1}} \cdots a_{k}^{i_{k}}s] &= \sum\limits_{j_{1} = 0}^{n} \cdots \sum\limits_{j_{k} = 0}^{n}P_{n}[S_{n-l} = a_{1}^{i_{1}} \cdots a_{k}^{i_{k}}s \mid S_{n} = a_{1}^{j_{1}} \cdots a_{k}^{j_{k}}s]\nu_{j_{1},...,j_{k}}\\
&= \sum\limits_{j_{1} = 0}^{n} \cdots \sum\limits_{j_{k} = 0}^{n} \frac{P_{n}[S_{n-l} = a_{1}^{i_{1}} \cdots a_{k}^{i_{k}}s, S_{n} = a_{1}^{j_{1}} \cdots a_{k}^{j_{k}}s]}{P_{n}[S_{n} = a_{1}^{j_{1}} \cdots a_{k}^{j_{k}}s]} \nu_{j_{1},...,j_{k}}\\
&=\sum\limits_{\forall x \in [1,k], i_{x} \leq j_{x} \leq l + \sum\limits_{y=1}^{k} i_{y} - \sum\limits_{y=1}^{x-1} j_{y}} \frac{\binom{l}{j_{1} - i_{1},...,j_{k} - i_{k}}\tilde{p}_{1}^{j_{1}}\cdots\tilde{p}_{k}^{j_{k}}}{\binom{n}{j_{1},...,j_{k}}\tilde{p}_{1}^{j_{1}}\cdots\tilde{p}_{k}^{j_{k}}} \nu_{j_{1},...,j_{k}}\\
&= \sum\limits_{\forall x \in [1,k], i_{x} \leq j_{x} \leq l + \sum\limits_{y=1}^{k} i_{y} - \sum\limits_{y=1}^{x-1} j_{y}} \frac{\binom{l}{j_{1} - i_{1},...,j_{k} - i_{k}}}{\binom{n}{j_{1},...,j_{k}}} \nu_{j_{1},...,j_{k}}
\end{align*}

where $\tilde{p}_{1}, ..., \tilde{p}_{k}$ represent the risk neutral probabilities and we have used the fact that we need $i_{x} \leq j_{x}, n - l - \sum\limits_{x=1}^{k-1} i_{x} \leq n - \sum\limits_{x=1}^{k-1} j_{x}$ for the numerators to be nonzero.  Next, 

$P_{n}^{\nu_{n}}[S_{n-l+1} = a_{1}S_{n-l}, S_{n-l} = a_{1}^{i_{1}} \cdots a_{k}^{i_{k}}s]  = $
\begin{align*}
 ... \\ &= \sum\limits_{j_{1} = 0}^{n} \cdots \sum\limits_{j_{k} = 0}^{n} \frac{P_{n}[S_{n-l+1} = a_{1}S_{n-l}, S_{n-l} = a_{1}^{i_{1}}\cdots a_{k}^{i_{k}}s, S_{n} = a_{1}^{j_{1}} \cdots a_{k}^{j_{k}}s]}{P_{n}[S_{n} = a_{1}^{j_{1}} \cdots a_{k}^{j_{k}}s]} \nu_{j_{1},...,j_{k}} \\ &= \sum\limits_{\forall x \in [1,k], i_{x} \leq j_{x} \leq l - 1 - \sum\limits_{y=1}^{k} i_{y} - \sum\limits_{y=1}^{x-1} j_{y}} \frac{\binom{l-1}{j_{1} - i_{1} - 1,...,j_{k} - i_{k}}}{\binom{n}{j_{1},...,j_{k}}} \nu_{j_{1},...,j_{k}}\\
\end{align*}
\\
Thus, \\
$P_{n}^{\nu_{n}}[S_{n-l+1} = a_{1}S_{n-l} \mid S_{n-l} = a_{1}^{i_{1}} \cdots a_{k}^{i_{k}}s] = \frac{\sum\limits_{\forall x \in [1,k], i_{x} \leq j_{x} \leq l - 1 - \sum\limits_{y=1}^{k} i_{y} - \sum\limits_{y=1}^{x-1} j_{y}} \frac{\binom{l}{j_{1} - i_{1},...,j_{k} - i_{k}}}{\binom{n}{j_{1},...,j_{k}}}}{\sum\limits_{\forall x \in [1,k], i_{x} \leq j_{x} \leq l - 1 - \sum\limits_{y=1}^{k} i_{y} - \sum\limits_{y=1}^{x-1} j_{y}} \frac{\binom{l-1}{j_{1} - i_{1} - 1,...,j_{k} - i_{k}}}{\binom{n}{j_{1},...,j_{k}}}} \nu_{j_{1},...,j_{k}}$
\\
\\
\end{proof}
\end{subsection}
\begin{subsection}{Markov Property of Trinomial Random Walk under Weak Anticipation}
\begin{proposition}
The process $(S_{k})_{0 \leq k \leq n}$ is Markov under $P_{n}^{\nu_{n}}$
\end{proposition}

\begin{proof} Let $T_{k} = \log(S_{k})$.  Then, since 
$\frac{S_{k+1}}{S_{k}} = \begin{cases}
a, \text{when } \omega_{k+1} = u & p = \tilde{p}_{1}\\
b, \text{when } \omega_{k+1} = m & p = \tilde{p}_{2}\\
c, \text{when } \omega_{k+1} = d & p = \tilde{p}_{3}
\end{cases}$
where $(\omega_{1},...,\omega_{k+1}) \in \left\{u, m, d\right\}^{k+1}$.  Thus,
\begin{align*}
\ E_{P_{n}^{\nu_{n}}}[e^{-\lambda(T_{k+1} - T_{k})} | \mathcal{F}_{k}] &= E_{P_{n}^{\nu_{n}}}[e^{-\lambda \log \left(\frac{S_{k+1}}{S_{k}}\right)} | \mathcal{F}_{k}] \\ &= a^{-\lambda}\tilde{p_{1}} + b^{-\lambda}\tilde{p}_{2} + c^{-\lambda}\tilde{p}_{3}\\
\end{align*}
Using inverse Laplace transforms we see $(T_{k+1} - T_{k})_{0 \leq k \leq n}$ are i.i.d. trinomial, and $\{T_{k}\}_{k \geq 1}$ is Markov.  Thus $(\frac{S_{k+1}}{S_{k}})_{0 \leq k \leq n}$ are i.i.d. and $\left(S_{k}\right)_{0 \leq k \leq n}$ is Markov.
\end{proof}
\end{subsection}
\end{section}

\begin{section}{Weak Anticipation and the Financial Value of Weak Information in the Continuous-Time Economy}
\begin{subsection}{Mathematical Framework}
Let $T > 0$ be a constant finite time horizon.  In this section we consider continuous-time, arbitrage-free economies of the form
\begin{align*}
(\Omega, \mathcal{H} = (\mathcal{H}_{t})_{0 \leq t \leq T}, P, S = (S_{t})_{0 \leq t \leq T})
\end{align*}
which constitutes a sample space $\Omega$, a filtration $\mathcal{H}$ that satisfies the usual conditions, a probability measure $P$.  When studying the strong information case, $\mathcal{H}$ at times represents the public information flow $\mathcal{F}$, and at times represent the enlarged filtration of the informed insider $\mathcal{G}$.  In either the strong or weak case, we assume the price process $S$ consists of d tradable assets, is an $\mathcal{F}$-adapted local martingale, i.e. $S_{t} \in \mathbb{R}^{d}$ and $S_{t} \in \mathcal{F}_{t}$ for each $0 \leq t \leq T$.  In addition, we assume $S$ is square-integrable and that $\mathcal{F}_{0}$ is trivial, which implies that $S_{0}$ is constant. \\ 
In the strong information case, we start with a Polish space $\mathcal{P}$, endowed with its Borel $\sigma$-algebra $\mathcal{B}(\mathcal{P})$, and a random variable $Y: \Omega \rightarrow \mathcal{P}$ which is $\mathcal{F}_{T}$-measurable.  Then we denote the information flow of the informed insider by filtration $\mathcal{G} = (\mathcal{G}_{t})_{0 \leq t < T}$ where $\mathcal{G}_{t} = \underset{\epsilon > 0}{\bigcap}(\mathcal{F}_{t + \epsilon} \vee \sigma(Y))$ for each $t < T$. \\
It is standard in the literature of strong information, see for example \cite{MR2021790} to denote $P_{Y}$ by the law of $Y$ and assume that $Y$ admits a regular disintegration with respect to the filtration $\mathcal{F}$.  So, the following assumption is made.
\begin{assumption}
There exists a jointly measurable, continuous in t, and $\mathcal{F}$-adapted process
\begin{align*}
\eta_{t}^{y}, \text{ } 0 \leq t < T, \text{ } y \in \mathcal{P},
\end{align*}
satifying for all $dt \bigotimes P_{Y}$ almost every $0 \leq t < t$ and $y \in \mathcal{P}$,
\begin{align*}
P[Y \in dy | \mathcal{F}_{t}] = \eta_{t}^{y} P[Y \in y]
\end{align*}
\end{assumption}
This assumption is not too restrictive, and will be satisfied for "nice" functionals, to be seen.  The existence of a conditional density $\frac{P[Y \in dy | \mathcal{F}_{t}]}{P[Y \in y]}$ is the main point, the existence of a regular version follows from general results on stochastic processes.
\begin{remark}
If we denote by $P^{y}$ the disintegrated probability measure defined by $P^{y} = P[\cdot | Y = y]$, then the above assumption implies that for $t < T$,
\begin{align*}
P_{/\mathcal{F}_{t}}^{y} = \eta_{t}^{y}P_{/\mathcal{F}_{t}}
\end{align*}
In particular, for $P_{Y}$-a.e. $y \in \mathcal{P}$, the process $(\eta_{t}^{y})_{0 \leq t < T}$ is a martingale in the filtration $\mathcal{F}$ (not uniformaly integrable).
\end{remark}
However, the aim of this paper is to replace this regular distintegration which tells us the anticipation path-by-path with the distribution of the anticipation, which characterizes the weak information case.  This leads us to introduce the analogous continuous-time versions of objects introduced in the second section of this paper.
\begin{definition}
The space $\mathcal{M}_{\mathcal{F}}(S)$ of martingale measures is the set of probability measures $\tilde{P} \sim P$ such that $(S_{t})_{0 \leq t \leq T}$ is an $\mathcal{F}$-adapted local martingale under $\tilde{P}$.
\end{definition}
\begin{definition}
The space $\mathcal{A}_{\mathcal{F}}(S)$ of admissible strategies is the space of $\mathbb{R}^{d}$-valued and $\mathcal{F}$ predictable processes $\Theta$ integrable with respect to the price process $S$, such that
\begin{align*}
\left(\int_{0}^{t} \Theta_{u} \cdot dS_{u} \right)_{0 \leq t \leq T}
\end{align*}
is a $(\tilde{P}, \mathcal{F})$ martingale for all $\tilde{P} \in \mathcal{M}_{\mathcal{F}}(S)$.
\end{definition}
Note that $\Theta_{t}^{i}$ is the number of shares of the risky asset $S_{t}^{i}$ held by an investor at time $t$, and the wealth process associated with the strategy $\Theta \in \mathcal{A}_{\mathcal{F}}(S)$, with initial capital x, is given by
\begin{align*}
V_{t} = x + \int_{0}^{t} \Theta_{u} \cdot dS_{u}
\end{align*}
In particular, we assume our strategies are self-financing. \\
We shall also often assume that the financial market
\begin{align*}
(\Omega, (\mathcal{F}_{t})_{0 \leq t \leq T}, P, (S_{t})_{0 \leq t \leq T})
\end{align*}
is complete in the sense that the martingale $(S_{t})_{0 \leq t \leq T}$ enjoys the following predictable representation property (PRP): For each $\mathcal{F}$-adapted local martingale $(M_{t})_{0 \leq t \leq T}$ there exists a unique predictable $\Theta$ locally in $L^{2}$ such that
\begin{align*}
M_{t} = M_{0} + \int_{0}^{t} \Theta_{u} \cdot dS_{u} \text{, for each } t \leq T
\end{align*}
\begin{remark}
Under the previous assumption, we have $\mathcal{M}_{\mathcal{F}}(S) = \{\tilde{P}\}$.
\end{remark}
Now let $\nu$ be a probability measure on $\mathcal{B}(\mathcal{P})$ corresponding to the anticipation of the informed insider.  We assume $\nu \sim P_{Y}$ with a bounded density $\xi$.  Intuititively this means that the informed investor does not have information which is "opposite" to the market. \\
Recall that the minimal information is $\nu = P_{Y}$ because under $P$ the price process $(S_{t})_{0 \leq t \leq T}$ is a local martingale and the maximal information is obtained in the limit with $\nu = \delta_{y}$ for $y \in \mathcal{P}$.
Let $\mathcal{E}^{\nu}$ be the set of probability measure $Q$ on $(\Omega, \mathcal{F}_{T})$ such that:
\begin{enumerate}
	\item $Q \sim P$
	\item $Q_{Y} = \nu$
\end{enumerate}
The financial market associated with an element $Q \in \mathcal{E}^{\nu}$ is 
\begin{align*}
(\Omega, (\mathcal{F}_{t})_{0 \leq t \leq T}, Q, (S_{t})_{0 \leq t \leq T})
\end{align*}
It is clear there is no arbitrage on this market because $Q \sim P$ and $S$ is a local martingale under $P$.  The portfolio optimization problem associated with this market is
\begin{align*}
\underset{\Theta \in \mathcal{A}_{\mathcal{F}}(S)}{\sup} E_{Q}\left[U\left(x + \int_{0}^{t} \Theta_{u} \cdot dS_{u}\right)\right]
\end{align*}
where $x > 0$ is the initial investment of the insider.
\begin{definition}
We define the \textbf{financial value of weak information} $(Y, \nu)$ for an insider with initial investment $x > 0$ as
\begin{align*}
u(x, \nu) := \underset{Q \in \mathcal{E}^{\nu}}{\inf}{\underset{\Theta \in \mathcal{A}_{\mathcal{F}}(S)}{\sup}}E_{Q}\left[U\left(x + \int_{0}^{T} \Theta_{u} \cdot dS_{u} \right)\right].
\end{align*}
\end{definition}
$u(x, \nu)$ is then the minimal gain in utility associated with the anticipation $(Y, \nu)$.  It is shows in \cite{MR2213259} that we always have $u(x, \nu) \geq U(x)$ and equality takes place for $\nu = P_{Y}$. \\
Remember we assume that S has the PRP.  We now definte the minimal probability associated with weak information.
\begin{definition}
The probability measure $P^{\nu}$ defined on $(\Omega, \mathcal{F}_{T})$ by
\begin{align*}
P^{\nu}(A) := \int_{\mathcal{P}} P[A | Y = y] \nu(dy) \text{, for all } A \in \mathcal{F}_{T}
\end{align*}
is called the \textbf{minimal probability associated with the weak information} $(Y, \nu)$.
\end{definition}
Intuitively, $P^{\nu}$ is constructed in order that for events A which are $P$-independent of $Y$ we have $P^{\nu}(A) = P(A)$.  Some immediate consequences of this definiton are:
\begin{enumerate}
	\item The law of $Y$ under $P^{\nu}$ is $\nu$
	\item We have the following relationship
		\begin{align*}
			dP^{\nu} = \xi(Y)dP \text{ recall } \xi = \frac{d\nu}{dP_{Y}}
		\end{align*}
\end{enumerate}
By Proposition of \cite{MR2213259}, $P^{\nu}$ is minimal amongst probability measures $Q \in \mathcal{E}^{\nu}$ on $\Omega$ which are equivalent to $P$ and such that the law of $Y$ under $Q$ is $\nu$.  In particular, we have the following.
\begin{theorem}
If $\phi$ is a convex function, then 
\begin{align*}
\underset{Q \in \mathcal{E}^{\nu}}{\inf}E_{P}\left[\phi\left(\frac{dQ}{dP}\right)\right] = E_{P}\left[\phi\left(\frac{dP^{\nu}}{dP}\right)\right] 
\end{align*}
\end{theorem}
Theorem 19 of \cite{MR2021790} gives us a closed form expression for the financial value of weak information.  Namely, we have the following result.
\begin{theorem}
Assume that the integrals below are convergent.  Then for each initial investment $x > 0$,
\begin{align*}
\ u(x,\nu) & = \underset{\Theta \in \mathcal{A}(S)}{\sup} E_{P^{\nu}} \left[U \left(x + \int_{0}^{T} \Theta_{u} dS_{u} \right) \right]
\\ &= \int_{\mathcal{P}} (U \circ I) \left (\frac{\Lambda(x)}{\xi(y)} \right) \nu (dy)
\end{align*}
where $\Lambda(x)$ is defined by 
\begin{align*}
\int_{\mathcal{P}} I \left(\frac{\Lambda(x)}{\xi(y)} \right) P_{Y}(dy) = x
\end{align*}
Moreover, under $P^{\nu}$ the optimal wealth process is given by 
\begin{align*}
V_{t} = \int_{\mathcal{P}} I \left(\frac{\Lambda(x)}{\xi(y)} \right) \eta_{t}^{y} P_{Y}(dy)
\end{align*}
and the corresponding number of parts invested in the risky asset S by 
\begin{align*}
\Theta_{t} = \int_{\mathcal{P}} I \left(\frac{\Lambda(x)}{\xi(y)} \right) \eta_{t}^{y} \alpha_{t}^{y} P_{Y}(dy)
\end{align*}
where we can choose $\alpha$ such that for $P_{Y}$ a.e. $y \in \mathcal{P}$ and for $0 \leq t < T$,
\begin{align*}
\eta_{t}^{y} = \exp \left( \int_{0}^{t} \alpha_{u}^{y} \cdot dS_{u} - \frac{1}{2} \int_{0}^{t} (\alpha_{u}^{y})^{*} d\left\langle S \right\rangle_{u} \alpha_{u}^{y} \right) \text{ on } \{\eta_{t}^{y} > 0\}
\end{align*}
\end{theorem}
\end{subsection}
\begin{subsection}{Donsker's Invariance Principle for Asymmetric Random Walk}
\begin{theorem}
Let $(\Omega, \mathcal{F}, P)$ be a probability space on which a sequence $\{\xi_{j}^{1}\}_{j=1}^{\infty}$ of independent, identically distributed random variables with variance $\sigma^{2}$ and

\[\xi_{j}^{1} = \begin{cases}
-1, & \text{with proability } 1-p \\
1, & \text{with probability } p > 0
\end{cases}\]

Define $Z^{(n)} = \{Z_{t}^{(n)}, t \geq 0\}$ where $Z_{t}^{(n)} = \frac{Y_{nt}^{1}}{\sigma \sqrt{n}}$ and $Y_{nt}^{1} = S_{\lfloor t \rfloor} + (t - \lfloor t \rfloor)\xi_{\lfloor t \rfloor + 1}^{1}, t \geq 0$.  Let $P_{n}$ be the measure induced by $Z^{(n)}$ on $(C[0,\infty), \mathcal{B}(C[0,\infty)))$. Then $\{P_{n}\}_{n=1}^{\infty}$ converges weakly to a measure $P_{\star}$ under which the coordinate mapping process $W_{t}(\omega) = w(t)$ on $C[0,\infty)$ is a standard, one-dimensional Brownian motion with drift $(2p-1)t$.

\end{theorem}

\begin{proof} Let $\xi_{j} = \xi_{j}^{1} - (2p - 1)$ for each $j \geq 1$.  Then $E[\xi_{j}] = 0$ and $Var[\xi_{j}] = pq = \sigma^{2} \geq 0$.  Set $X_{t}^{(n)} = \frac{Y_{nt}}{\sigma \sqrt{n}}$ where $Y_{t} = S_{\lfloor t \rfloor} + (t - \lfloor t \rfloor)\xi_{\lfloor t \rfloor + 1}$ as in (4.10) of \cite{MR1121940}.  By Theorem 2.4.17, p. 67 of \cite{MR1121940} $(X_{t_{1}}^{(n)}, ... , X_{t_{d}}^{(n)}) \overset{d}{\rightarrow} (B_{t_{1}}, ... , B_{t_{d}})$ as $n \rightarrow \infty$, where $\{B_{t}\}_{t \geq 0}$ denotes a standard one-dimensional Brownian motion.  By a corollary of \cite{MR1796326} on p. 109, we see convergence in distribution is preserved by addition or multiplication by a convergent sequence of constants, i.e. $(Z_{t_{1}}^{(n)}, ... , Z_{t_{d}}^{(n)}) \overset{d}{\rightarrow} (B_{t_{1}} + (2p-1)t_{1}, ... , B_{t_{d}} + (2p - 1)t_{d})$ as $n \rightarrow \infty$.  That $Z^{(n)}$ is tight follows exactly as in the proof of Donsker's Theorem 2.4.20, p. 71 of \cite{MR1121940} as 
\begin{align*}
\lim_{\delta \downarrow 0} \sup_{n \geq 1}P\left[\max_{|s - t| \leq \delta, 0 \leq s,t \leq T} |Z_{t}^{(n)} - Z_{s}^{(n)}| > \epsilon \right] &\leq \lim_{\delta \downarrow 0} \sup_{n \geq 1} P\left[\max_{|s - t| \leq \delta, 0 \leq s,t \leq T} |X_{t}^{(n)} - X_{s}^{(n)}| > \frac{\epsilon}{2} \right] \\ &+ \lim_{\delta \downarrow 0} \sup_{n \geq 1} P\left[\max_{|s - t| \leq \delta, 0 \leq s,t \leq T} (2p - 1)|t - s| > \frac{\epsilon}{2} \right] \\ &= 0
\end{align*}

By Theorems 2.4.15, 2.4.17 of \cite{MR1121940}, $Z_{t}^{(n)} \overset{d}{\rightarrow} B_{t} + (2p - 1)t$ as $n \rightarrow \infty$.
\end{proof}

\textbf{Note:}  If we instead let \[\xi_{j}^{1} = \begin{cases}
0, & \text{with probability } 1-p \\
1, & \text{with probability } p > 0
\end{cases}\]
in the previous theorem, then by repeating the proof we will obtain $Z_{t}^{(n)} \overset{d}{\rightarrow} B_{t} + pt$ as $n \rightarrow \infty$.
\end{subsection}
\begin{subsection}{Weak convergence of multinomial asymmetric random walks to Brown motions with drift}
For the main proposition of this subsection, we will need two famous results which can be found in \cite{MR2512800}, for instance.
\begin{theorem}\textbf{(Glivenko)}
If $\phi_{n}$ and $\phi$ are the characteristic functions of probability distributions $P_{n}$ and $P$ (respectively), for each $n \in \mathbb{N}$, then $\phi_{n}(u) \rightarrow \phi(u)$ for all $u \in \mathbb{R}^{d}$ implies $P_{n} \overset{d}{\rightarrow} P$ as $n \rightarrow \infty$.
\end{theorem}
\begin{theorem}\textbf{(Levy's Continuity Theorem)}
If $\{\phi_{n}\}_{n \in \mathbb{N}}$ is a sequence of characteristic functions and there exists a function $\psi : \mathbb{R}^{d} \rightarrow \mathbb{C}$ such that, for all $u \in \mathbb{R}^{d}$, $\phi_{n}(u) \rightarrow \psi(u)$ as $n \rightarrow \infty$ as $n \rightarrow \infty$ and $\psi$ is continuous at 0, then $\psi$ is the characteristic function of a probability distribution.
\end{theorem}
\begin{theorem}\textbf{(Kac's Theorem)}
The random variables $X_{1},...,X_{n}$ are independent if and only if
\begin{align*}
E\left[ \exp \left(i \sum_{j=1}^{n}(u_{j},X_{j}) \right) \right] = \phi_{X_{1}}(u_{1}) \cdots \phi_{X_{n}}(u_{n})
\end{align*}
for all $u_{1},...,u_{n} \in \mathbb{R}^{d}$.
\end{theorem}
\begin{proposition}
The Trinomial Asymetric Random Walk converges in distribution to the sum of two independent Brownian motions with drift.
\end{proposition}

\begin{proof}
Write the trinomial random walk as $\xi_{n} = \xi_{n}^{1} + \xi_{n}^{2} - 1$ where 
\[\xi_{j}^{1} = \begin{cases}
0, & \text{with probability } 1-p \\
1, & \text{with probability } p > 0
\end{cases}\]
and
\[\xi_{j}^{2} = \begin{cases}
0, & \text{with probability } 1-q \\
1, & \text{with probability } q > 0
\end{cases}\]
and $\xi_{m}^{i} \perp \xi_{n}^{j}$ for all $i, j, m, n$.  Then
\[\xi_{n} = \begin{cases}
1, & \text{with probability } pq \\
0, & \text{with probability } p(1-q) + (1-p)q \\
-1, & \text{with probability } (1-p)(1-q)
\end{cases}\]
Write $Z_{t}^{(n)} = \frac{Y_{nt}^{1}}{\sigma_{1} \sqrt{n}} + \frac{Y_{nt}^{2}}{\sigma_{2} \sqrt{n}}$, where 
\begin{align}
Y_{nt}^{1} = S_{\lfloor t \rfloor}^{1} + (t - \lfloor t \rfloor)\left(\xi_{\lfloor t \rfloor + 1}^{1} - \frac{1}{2}\right) \\
Y_{nt}^{2} = S_{\lfloor t \rfloor}^{2} + (t - \lfloor t \rfloor)\left(\xi_{\lfloor t \rfloor + 1}^{2} - \frac{1}{2}\right)
\end{align}
and 
\begin{align}
S_{\lfloor t \rfloor}^{1} = \sum\limits_{j = 0}^{\lfloor t \rfloor} \left(\xi_{j}^{1} - \frac{1}{2}\right) \\
S_{\lfloor t \rfloor}^{2} = \sum\limits_{j = 0}^{\lfloor t \rfloor} \left(\xi_{j}^{2} - \frac{1}{2}\right)
\end{align}
Note $Y_{nt}^{1} \perp Y_{nt}^{2}$.  So, imitating the Donsker principle for an asymmetric random walk, the summands satisfy
\[\xi_{j}^{1} - \frac{1}{2} \overset{d}{=} \begin{cases}
\frac{-1}{2}, &\text{with probability } 1-p \\
\frac{1}{2}, &\text{with probability } p
\end{cases}\]
and
\[\xi_{j}^{2} - \frac{1}{2} \overset{d}{=} \begin{cases}
\frac{-1}{2}, &\text{with probability } 1-q \\
\frac{1}{2}, &\text{with probability } q
\end{cases}\]
We see 
\begin{align}
\frac{Y_{nt}^{1}}{\sigma_{1}\sqrt{n}} \overset{d}{\rightarrow} B_{t}^{1} + \left(p - \frac{1}{2}\right)t \\
\frac{Y_{nt}^{2}}{\sigma_{2}\sqrt{n}} \overset{d}{\rightarrow} B_{t}^{2} + \left(q - \frac{1}{2}\right)t
\end{align}
Using $Y_{nt}^{1} \perp Y_{nt}^{2}$, we see
\begin{align*}
\exp(iuZ_{t}^{(n)}) &= \exp\left(iu\left(\frac{Y_{nt}^{1}}{\sigma_{1}\sqrt{n}} + \frac{Y_{nt}^{2}}{\sigma_{2}\sqrt{n}}\right)\right) \\
&= \exp\left(iu\frac{Y_{nt}^{1}}{\sigma_{1}\sqrt{n}}\right)\cdot \exp\left(iu\frac{Y_{nt}^{2}}{\sigma_{2}\sqrt{n}}\right) \text{ by Kac } \\
& \rightarrow \exp\left(iu\left(B_{t}^{1} + \left(p - \frac{1}{2}\right)t\right)\right) \cdot \exp\left(iu\left(B_{t}^{2} + \left(q - \frac{1}{2}\right)t\right)\right) \text{ by Glivenko }\\
&= \exp\left(iu\left(p + q - 1\right)t -\frac{1}{2}t^{2}u^{2}\right)
\end{align*}
Since the last expression is the characteristic function of the sum of two independent Brownian motions with drift, as well as a function continuous 0, we may apply Levy's continuity theorem to obtain $Z_{t}^{(n)} \overset{d}{\rightarrow} B_{t}^{1} + B_{t}^{2} + (p + q - 1)t$ where $B_{t}^{1} \perp B_{t}^{2}$.  By multiplying the summands by constants or adding constants we achieve the same result, with only the drift term affected by the corollary of \cite{MR1796326} on p. 109.
\end{proof}

\begin{proposition}
The n-multinomial asymmetric random walk converges in distribution to the sum of $2^{(n-2)}$ independent standard Brownian motions with drift.
\end{proposition}

\begin{proof}
The previous theorem is the base case for our induction argument.  Suppose this is true for the $(n-1)$-multinomial asymmetric random walk.  Then since the summands of a $n$-multinomial asymmetric random walk can be written as the sum of $(n-1)$-multinomial asymmetric random walks, we see the n-multinomial converges to the sum of $2 \cdot 2^{(n-3)}$ independent standard Brownian motions with drift.
\end{proof}
\end{subsection}
\end{section}

\begin{section}{The Main Result (Complete Case)}
The goal of this section is to prove that under certain conditions, we can approximate the optimal utility by a sequence of discretized versions.  More clearly, we wish to prove $u(x, \nu_{n}) \rightarrow u(x, \nu)$ for each $x > 0$, where we replace $\xi(y) = \frac{d\nu}{dP_{Y}}(y)$ by $\xi_{n}(y) = \frac{d\nu_{n}}{dP_{Y_{n}}}(y)$ in $u(x, \nu_{n})$.  In order to do this, we need some results from previous works, and some assumptions.  We assume that $\mathcal{P}_{n} \subseteq \mathcal{P}$ for each $n \geq 0$ in some sense.  For example, if $\mathcal{P}_{n}$ is the range of a symmetric binomial random walk, then we can embed this into $\mathcal{P} = C_{0}[0,T]$, the space of continuous functions starting at 0, as in \cite{MR3971211}, for example.\\

Next, we need the following lemma which is a generalization of Problem 2.4.12 of \cite{MR1121940}

\begin{lemma}
Let $\{f_{n}\}_{n = 0}^{\infty} \subset C(\mathcal{P})$ and $\{\mu_{n}\}_{n=0}^{\infty}$ be a sequence of probability measures such that
\begin{enumerate}
	\item $f_{n} \rightarrow f_{0}$ uniformly on compact sets
	\item $\underset{n \geq 0}{\sup}||f_{n}||_{\infty} =: B < \infty$
	\item $\mu_{n} \rightarrow \mu_{0}$ vaguely as $n \rightarrow \infty$
\end{enumerate}
Then 
\begin{align*}
\int f_{n}d\mu_{n} \rightarrow \int fd\mu \text{ as } n \rightarrow \infty
\end{align*}
\end{lemma}

\begin{proof}
Let $\epsilon > 0$, and choose a compact set $K \subset \mathcal{P}$ such that $\underset{n \geq 0}{\sup} \mu_{n} (K^{c}) < \epsilon /6B$ by Prohorov's Theorem and the tightness of $\{\mu_{n}\}_{n = 0}^{\infty}$.  Then, since $f_{n} \rightarrow f$ uniformly on compact sets we may choose $N \in \mathbb{N}$ such that $n \geq N$ implies $\underset{y \in K}{\sup} |f_{n}(y) - f(y)| < \epsilon / 3$.  Then 
\begin{align*}
\left| \int f_{n} d\mu_{n} - \int f d\mu \right| & \leq \left| \int f_{n} d\mu_{n} - \int fd\mu_{n} \right| + \left| \int f d\mu_{n} - \int fd\mu \right| \\ & \leq \int_{K} \underset{y \in K}{\sup} |f_{n}(y) - f(y)| d\mu_{n} + 2B\underset{n \geq 0}{\sup} \mu_{n}(K^{c}) + \left| \int f d\mu_{n} - \int f d\mu \right| \\ & \leq \underset{y \in K}{\sup} |f_{n}(y) - f(y)| + 2B\underset{n \geq 0}{\sup} \mu_{n}(K^{c}) + \left| \int f d\mu_{n} - \int f d\mu \right|
\end{align*}
Since $\mu_{n} \rightarrow \mu_{0}$ vaguely, we may choose $N' >> N$ such that $n \geq N'$ imples
\begin{align*}
\left| \int f d\mu_{n} - \int f d\mu \right| < \epsilon / 3
\end{align*}
Putting all these estimates together, we see
\begin{align*}
\lim \underset{n \rightarrow \infty}{\sup} \left| \int f_{n} d\mu_{n} - \int f d\mu \right| < \epsilon
\end{align*}
Since $\epsilon$ was arbitrary, we see 
\begin{align*}
\int f_{n} d\mu_{n} \rightarrow \int f d\mu \text{ as } n \rightarrow \infty .
\end{align*}
\end{proof}

Recall we may write an anticipation measure as $dP^{\nu} = \xi(Y)dP$, where $\xi = \frac{d\nu}{dP_{Y}}$.

\begin{proposition}
Let $P_{n}^{\nu_{n}}$ denote a sequence of anticipation measures where $P_{n}$ denotes the measure of the $n$th partial sum of a symmetric binomial random walk, $P$ denotes the Wiener measure, and write $\xi_{n} = \frac{d\nu_{n}}{dP_{Y}^{n}}$ and $\xi = \frac{d\nu}{dP_{Y}}$.  If the $\{\xi_{n}\}_{n = 1}^{\infty}$ are uniformly bounded with $\xi_{n} \rightarrow \xi$ uniformly on compact sets, then $P_{n}^{\nu_{n}} \overset{d}{\rightarrow} P^{\nu}$ as $n \rightarrow \infty$.
\end{proposition}

\begin{proof} 
By Donsker's Invariance Principle we know $P_{n} \overset{d}{\rightarrow} P$ as $n \rightarrow \infty$ where $P$ denotes the distribution of a standard one-dimensional Brownian motion.  Thus, by vague convergence of the $\{P_{n}\}_{n = 1}^{\infty}$, we further obtain $P_{n}^{\nu_{n}} \overset{d}{\rightarrow} P^{\nu}$ as $n \rightarrow \infty$.
\end{proof}

By using the same argument above with Donsker's Invariance Principle for the symmetric Trinomial Random Walk and symmetric Multinomial Random Walk respectively, we obtain the folllowing propositions.

\begin{proposition}
Let $P_{n}^{\nu_{n}}$ denote a sequence of anticipation measures where $P_{n}$ denotes the measure of the $n$th partial sum of a symmetric trinomial random walk, $P$ denotes the distribution of a sum of two independent standard one-dimensional Brownian motions, and write $\xi_{n} = \frac{d\nu_{n}}{dP_{Y}^{n}}$ and $\xi = \frac{d\nu}{dP_{Y}}$.  If the $\{\xi_{n}\}_{n = 1}^{\infty}$ are uniformly bounded with $\xi_{n} \rightarrow \xi$ uniformly on compact sets, then $P_{n}^{\nu_{n}} \overset{d}{\rightarrow} P^{\nu}$ as $n \rightarrow \infty$.
\end{proposition}

\begin{proposition}
Let $P_{n}^{\nu_{n}}$ denote a sequence of anticipation measures where $P_{n}$ denotes the measure of the $n$th partial sum of a symmetric N-multinomial random walk, $P$ denotes the distribution of a sum of $2^{N-2}$ independent standard one-dimensional Brownian motions, and write $\xi_{n} = \frac{d\nu_{n}}{dP_{Y}^{n}}$ and $\xi = \frac{d\nu}{dP_{Y}}$.  If the $\{\xi_{n}\}_{n = 1}^{\infty}$ are uniformly bounded with $\xi_{n} \rightarrow \xi$ uniformly on compact sets, then $P_{n}^{\nu_{n}} \overset{d}{\rightarrow} P^{\nu}$ as $n \rightarrow \infty$.
\end{proposition}

\begin{theorem}
Recall $u(x, \nu_{n}) = \int_{\mathcal{P}} (U \circ I) \left( \frac{\Lambda_{n}(x)}{\xi_{n}(y)} \right) \nu_{n}(dy)$. Let $\xi = \xi_{0}$. If 
\begin{enumerate}
	\item $U \circ I \in C(\mathbb{R})$
	\item $\frac{\Lambda_{n}(x)}{\xi_{n}(y)} \rightarrow \frac{\Lambda(x)}{\xi(y)}$ uniformly on compact sets as $n \rightarrow \infty$ and $\underset{n \geq 0}{\sup} ||\xi_{n}||_{\infty} < \infty$
	\item $\nu_{n} \rightarrow \nu$ vaguely as $n \rightarrow \infty$
\end{enumerate}

then $u(x, \nu_{n}) \rightarrow u(x, \nu)$ as $n \rightarrow \infty$ for any $x >0$.
\end{theorem}

\begin{proof}
If we let $f_{n}(y) = \left( U \circ I \right) \left( \frac{\Lambda_{n}(x)}{\xi_{n}(y)} \right)$ and $\mu_{n} = \nu_{n}$ then the hypotheses of the previous lemma are satisfied so we obtain the result.
\end{proof}

\begin{example}
Let $\alpha \in (0, 1)$, $U(x) = \frac{x^{\alpha}}{\alpha}$ so
\begin{align*}
u(x, \nu_{n}) = \frac{x^{\alpha}}{\alpha}\left[\int_{\mathcal{P}}\left(\frac{d\nu_{n}}{dP_{Y_{n}}}(y)\right)^{\frac{1}{1-\alpha}}P[Y_{n} \in dy]\right]^{1-\alpha}
\end{align*}
Recall
\begin{align*}
\Lambda(x) = \frac{x^{\alpha - 1}}{\left( \int_{\mathcal{P}} \left( \frac{d\nu}{dP_{Y}} \right)^{\frac{1}{1 - \alpha}} P[Y \in dy] \right)^{\alpha - 1}}
\end{align*}
since $U'(x) = x^{\alpha - 1}$ so $I(x) = x^{\frac{1}{\alpha - 1}}$.  So if the hypotheses of our theorem hold, then 
\begin{align*}
u(x, \nu_{n}) \rightarrow u(x, \nu) = \frac{x^{\alpha}}{\alpha}\left[\int_{\mathcal{P}}\left(\frac{d\nu}{dP_{Y}}(y)\right)^{\frac{1}{1-\alpha}}P[Y \in dy]\right]^{1-\alpha}
\end{align*}
\end{example}

\begin{example}
Let $U(x) = \ln(x)$ so
\begin{align*}
u(x, \nu_{n}) = \ln(x) + \int_{\mathcal{P}} \frac{d\nu_{n}}{dP_{Y_{n}}}(y) \ln \left( \frac{d\nu_{n}}{dP_{Y_{n}}}(y) \right) P[Y_{n} \in dy]
\end{align*}
Recall
\begin{align*}
\Lambda(x) = \frac{1}{x}
\end{align*}
since $U'(x) = \frac{1}{x}$ so $I(x) = \frac{1}{x}$.  So if the hypotheses of our theorem hold, then
\begin{align*}
u(x, \nu_{n}) \rightarrow u(x, \nu) =  \ln(x) + \int_{\mathcal{P}} \frac{d\nu}{dP_{Y}}(y) \ln \left( \frac{d\nu}{dP_{Y}}(y) \right) P[Y \in dy]
\end{align*}
\end{example}
\end{section}

\begin{section}{Conclusion}
In this paper, we expanded upon the results in \cite{MR2213259} in providing a model for weak anticipations in financial markets.  Note, it is also shown in \cite{MR1919610, MR2021790} that the results related to the theory of initial enlargement can be recovered from the weak anticipation approach by letting the anticipating measure $\nu = \delta_{Y}$.
\\
We defined weak anticipation in the discrete-time economy in terms of the minimal probability associated with weak information, $P_{n}^{\nu_{n}}$, as well as the financial value of weak information in the discrete-time economy, $u(x, \nu_{n})$.  Then, for each $n \in \mathbb{N}$ and each $n$-th discrete-time economy, we calculated the transition probabilities of $P_{n}^{\nu_{n}}$ when the price process $S^{n}$ was trinomial or multinomial.  We also showed that under these transition probabilities, $S^{n}$ was Markov.
\\
Next, we defined weak anticipation in the continuous-time economy in terms of the minimal probability measure associated with weak information, $P^{\nu}$, as well as the financial value of weak information in the continuous-time economy $u(x, \nu)$.  The main result of this work then stated that for each initial endowment $x > 0$, under certain conditions we have $u(x, \nu_{n}) \rightarrow u(x, \nu)$ as $n \rightarrow \infty$, giving relevant examples.
\\
Future results may focus on the incomplete case to prove convergence of the financial value of weak anticipation.  In particular, a good starting point may be Theorem 6 of \cite{MR2213259}, which provides an explicit formular for the financial value of weak information.  An example that may prove useful would be studying the discrete-time economy governed by the trinomial random walk illustrated in this paper, with continuous-time economy governed by the sum of independent Brownian motions with drift, and with terminal signal $\nu$ given by the terminal value of the risky asset $S_{T}$.
\end{section}

\bibliography{Lindsell_2022_bib}
\bibliographystyle{plain}
\nocite{*}
\end{document}